\documentclass[12pt,final]{article}
\usepackage[margin=2.5cm,top=1.5cm]{geometry}
\usepackage[all,arc]{xy}
\usepackage{mathrsfs}
\usepackage{amsmath,amsthm,amssymb,enumerate}
\usepackage{color}

\newtheorem{definition}{Definition}[section]
\newtheorem{lemma}[definition]{Lemma}
\newtheorem{theorem}[definition]{Theorem}
\newtheorem{corollary}[definition]{Corollary}

\newtheorem{proposition}[definition]{Proposition}

\newtheorem{con}[definition]{Conjecture}

\newtheorem{remark}[definition]{Remark}

\newcommand{\Romannum}[1]{\uppercase\expandafter{\romannumeral #1}}

\numberwithin{equation}{section}

\newcommand\keywordsname{Key words}
\newcommand\AMSname{AMS subject classifications}

\newenvironment{@abssec}[1]{%
     \if@twocolumn
       \section*{#1}%
     \else
       \vspace{.05in}\footnotesize
       \parindent .2in
         {\upshape\bfseries #1. }\ignorespaces
     \fi}
     {\if@twocolumn\else\par\vspace{.1in}\fi}

\begin{document}

\title{A conjecture on the primitive degree  of Tensors
\footnote{P. Yuan's research is supported by the NSF of China (Grant No. 11271142) and
the Guangdong Provincial Natural Science Foundation(Grant No. S2012010009942),
L. You's research is supported by the Zhujiang Technology New Star Foundation of
Guangzhou (Grant No. 2011J2200090) and Program on International Cooperation and Innovation, Department of Education,
Guangdong Province (Grant No.2012gjhz0007).}}
\author{Pingzhi Yuan\footnote{{\it{Corresponding author:\;}}yuanpz@scnu.edu.cn.}, Zilong He,  Lihua You \footnote{{\it{Email address:\;}}ylhua@scnu.edu.cn.}}
\vskip.2cm
\date{{\small
School of Mathematical Sciences, South China Normal University,\\
Guangzhou, 510631, P.R. China\\
}}
\maketitle

\begin{abstract} In this paper, we prove: Let $\mathbb{A}$ be a nonnegative primitive tensor with order $m$ and dimension $n$. Then its primitive degree $\gamma(\mathbb{A})\le (n-1)^2+1$, and the upper bound is sharp. This confirms a conjecture  of Shao \cite{Sh12}.
\end{abstract}

\section{Introduction}
\hskip.6cm In \cite{Ch1} and \cite{Ch2}, Chang et al investigated the properties of the spectra of nonnegative tensors. They defined the irreducibility of tensors, and the primitivity of nonnegative tensors, and extended many important properties of primitive matrices to primitive tensors. Recently, as an application of the general tensor product defined by Shao \cite{Sh12}, Shao presented a simple characterization of the primitive tensors in terms of the zero pattern of the powers of  $\mathbb{A}$. He also proposed the following conjecture on the primitive degree.

\begin{con} When $m$ is fixed, then there exists some polynomial $f(n)$ on $n$ such that $\gamma(\mathbb{A})\le f(n)$ for all nonnegative primitive tensors of order $m$ and dimension $n$. \end{con}

In this paper, we confirm the conjecture by proving the following theorem.

\begin{theorem}\label{thmm} Let $\mathbb{A}$ be a nonnegative primitive tensor with order $m$ and dimension $n$. Then its primitive degree $\gamma(\mathbb{A})\le (n-1)^2+1$, and the upper bound is sharp. \end{theorem}

\section{Preliminaries}
\hskip.6cm An order $m$ dimension $n$ tensor $\mathbb{A}= (a_{i_1i_2\cdots i_m})_{1\le i_j\le n \hskip.2cm (j=1, \cdots, m)}$ over the complex field $\mathbb{C}$ is a multidimensional array with all entries $a_{i_1i_2\cdots i_m}\in\mathbb{C}\, ( i_1, \cdots, i_m\in [n]=\{1, \cdots, n\})$. The majorization matrix  $M(\mathbb{A})$ of the tensor $\mathbb{A}$ is defined as  $(M(\mathbb{A}))_{ij}=
a_{ij\cdots j}, (i, j\in[n])$ by Pearson \cite{Pe10}.

 Let $\mathbb{A}$ (and $\mathbb{B}$) be an order $m\ge2$ (and $k\ge 1$), dimension $n$ tensor, respectively. Recently, Shao \cite{Sh12} defined a general product  $\mathbb{A}\mathbb{B}$ to be the following tensor $\mathbb{D}$ of order $(m-1)(k-1)+1$ and dimension $n$:
$$ d_{i\alpha_1\cdots\alpha_{m-1}}=\sum\limits_{i_2, \cdots, i_m=1}^na_{ii_2\cdots i_m}b_{i_2\alpha_1}\cdots b_{i_m\alpha_{m-1}} \quad (i\in[n], \, \alpha_1, \cdots, \alpha_{m-1}\in[n]^{k-1}).$$

The tensor product  possesses a very useful property: the associative law (\cite{Sh12}, Theorem 1.1).  
With the general product,  Shao \cite{Sh12} proved some results on the primitivity and primitive degree f nonnegative tensor. 
 The following result will be used in Definition \ref{defn22}.

\begin{proposition}{\rm(\cite{Sh12}, Proposition 4.1)} \label{Pro21} 
Let $\mathbb{A}$ be an order $m$ and dimension $n$ nonnegative  tensor. Then the following three conditions are equivalent:

(1). For any $i, j\in[n], a_{ij\cdots j}>0$ holds.

(2). For any $j\in[n],  \mathbb{A}e_j>0$ holds (where $e_j$ is the $j^{th}$ column of the identity matrix $I_n$).

(3). For any nonnegative nonzero vector $x\in \mathbb{R}^n, \mathbb{A}x>0$ holds. \end{proposition}

\begin{definition} {\rm (\cite{Pe10}, Definition 3.1)} \label{defn220}A nonnegative tensor $\mathbb{A}$ is  called essentially positive, if it satisfies (3) of Proposition \ref{Pro21}.\end{definition}

By Proposition \ref{Pro21}, the following Definition \ref{defn22} is eauivalent to Definition \ref{defn220}.

\begin{definition} {\rm (\cite{Sh12}, Definition 4.1)} \label{defn22}
A nonnegative tensor $\mathbb{A}$ is  called essentially positive, if it satisfies one of the three conditions in Proposition \ref{Pro21}.\end{definition}

In \cite{Ch2} and \cite{Pe10}, Chang et al and Pearson define the primitive tensors as follows.

\begin{definition} {\rm (\cite{Ch2, Pe102}) \label{defn23}} Let $\mathbb{A}$ be a nonnegative  tensor with order $m$ and dimension $n$,
$x=(x_1, x_2, \cdots, x_n)^T\in\mathbb{R}^n$ a vector and $x^{[r]}=(x_1^r, x_2^r, \cdots, x_n^r)^T$. Define the map $T_\mathbb{A}$ from $\mathbb{R}^n$ to $\mathbb{R}^n$ as: $T_\mathbb{A}(x)=(\mathbb{A}x)^{[\frac{1}{m-1}]}$. If there exists some positive integer $r$ such that $T_\mathbb{A}^r(x)>0$ for all nonnegative nonzero vectors $x\in\mathbb{R}^n$, then $\mathbb{A}$ is called primitive and the smallest such integer $r$ is called the primitive degree of $\mathbb{A}$, denoted by $\gamma(\mathbb{A})$. \end{definition}

In \cite{Sh12}, Shao show the following results and define the primitive degree by using the properties of tensor product and the zero patterns.

\begin{proposition}\label{Pro24}{\rm (\cite{Sh12}, Theorem 4.1)}
A nonnegative tensor $\mathbb{A}$ is primitive if and only if there exists some positive integer $r$ such that $\mathbb{A}^r$ is essentially positive.
Furthermore, the smallest such $r$ is the primitive degree of $\mathbb{A}$.
\end{proposition}

\begin{remark}\label{rem1}
Let $\mathbb{A}$ be a nonnegative  tensor with order $m$ and dimension $n$. Then $\mathbb{A}$ is primitive if and only if there exists some positive integer $r$ such that $M(\mathbb{A}^r)>0.$
\end{remark}

Now we prove the following necessary conditions for a tensor to be  primitive firstly.

\begin{proposition}\label{Pro25}Let $\mathbb{A}$ be a nonnegative primitive tensor with order $m$ and dimension $n$, $M(\mathbb{A})$ the majorization matrix of $\mathbb{A}$. Then we have:

{\rm (i). } For each $j\in[n]$, there exists an integer $i\in[n]\backslash\{j\}$ such that $(M(\mathbb{A}))_{ij}>0$.

{\rm (ii). } There exist some $j\in[n]$ and integers $u, v$ with $1\le u< v\le n$ such that  $(M(\mathbb{A}))_{uj}>0$ and $(M(\mathbb{A}))_{vj}>0$.\end{proposition}

\begin{proof} We prove the results via contradiction.

(i) To obtain a contradiction, we suppose that there exists some integer $j\in [n]$ such that $(M(\mathbb{A}))_{ij}=0$ for every $i\in[n]\backslash\{j\}$.
By definitions of the tensor product and the majorization matrix,  for any $u\in[n]\backslash\{j\}$, we have $(M(\mathbb{A}))_{uj}=0$ and

$(M(\mathbb{A}^2))_{uj}= \sum\limits_{j_2,\cdots, j_m=1}^{n}a_{uj_2\cdots j_m}a_{j_2j\cdots j}\cdots a_{j_mj\cdots j}$

\hskip2.0cm $=\sum\limits_{j_2,\cdots, j_m=1}^{n}a_{uj_2\cdots j_m}(M(\mathbb{A}))_{j_2j}\cdots(M(\mathbb{A}))_{j_mj}$

\hskip2.0cm $=(M(\mathbb{A}))_{uj}(M(\mathbb{A}))_{jj}^{m-1}=0$.

Since
$$(M(\mathbb{A}^{r+1}))_{uj}=\sum\limits_{j_2,\cdots, j_m=1}^{n}a_{uj_2\cdots j_m}(M(\mathbb{A}^r))_{j_2j}\cdots(M(\mathbb{A}^r))_{j_mj},$$hence, by induction on $k$, we conclude that
$$(M(\mathbb{A}^k))_{uj}=0$$ holds for any positive integer $k$ and any $u\in[n]\backslash\{j\}$. This contradicts that $(M(\mathbb{A}^{\gamma(\mathbb{A})}))_{uj}>0$, where $\gamma(\mathbb{A})$ is the primitive degree of $\mathbb{A}$. (i) is proved.

(ii) Suppose, to derive a contradiction, that there is at most one nonzero element in each of the following $n$ sets
$$\{(M(\mathbb{A}))_{uj}, u\in[n]\}, \quad j\in[n].$$
Now we will show that for any positive integer $t$, there is at most one nonzero element in each of the following $n$ sets
$$\{(M(\mathbb{A}^t))_{uj}, u\in[n]\}, \quad j\in[n].$$

We prove the above assertion by induction on $t$. Clearly, $t=1$ is obvious. Assume that the assertion holds for $t=k\ge1$, that is, there is at most one nonzero element in each of the following $n$ sets
$$\{(M(\mathbb{A}^k))_{uj}, u\in[n]\}, \quad j\in[n],$$
say, for any $j\in[n]$ and any $u\ne u_j$,  $(M(\mathbb{A}^k))_{uj}=0$. Note that for any $v\in [n]$,
$$(M(\mathbb{A}^{k+1}))_{vj}=\sum\limits_{j_2,\cdots, j_m=1}^{n}a_{vj_2\cdots j_m}(M(\mathbb{A}^k))_{j_2j}\cdots(M(\mathbb{A}^k))_{j_mj}
=a_{vu_j\cdots u_j}((M(\mathbb{A}^k))_{u_jj})^{m-1},$$
by the assumption, there is at most one $v\in[n]$ such that $a_{vu_j\cdots u_j}=(M(A))_{vu_j}>0$, therefore we have prove the assertion. It follows that $\mathbb{A}$ is not a primitive tensor, this contradiction proves (ii).\end{proof}

Follows from the proof of Proposition \ref{Pro25}, we know that
\begin{equation}\label{eqm1} (M(\mathbb{A}^{k+1}))_{uj}=\sum\limits_{j_2,\cdots, j_m=1}^{n}a_{uj_2\cdots j_m}(M(\mathbb{A}^k))_{j_2j}\cdots(M(\mathbb{A}^k))_{j_mj},\end{equation}
is useful and used repeatly.

 By Equation (\ref{eqm1}), it is easy to prove the following assertions.

\begin{corollary}\label{Cor26} Let $\mathbb{A}$ be a nonnegative primitive tensor with order $m$ and dimension $n$. Then we have

{\rm (i). } For each $u\in[n]$, there is at least one index $j_2\cdots j_m\in[n]^{m-1}$ such that $a_{uj_2\cdots j_m}>0$.

{\rm (ii). } Let $k\in[n]$ be fixed. Suppose that $T$ is a positive integer such that $(M(\mathbb{A}^T))_{uk}>0$ for all $u\in[n]$, then for any $t\ge T$, we have $(M(\mathbb{A}^t))_{uk}>0$ for all $u\in[n]$.\end{corollary}

\begin{proposition}\label{Pro27}
 Let $\mathbb{A}$ be a nonnegative  tensor with order $m$ and dimension $n$, and let  $a$ be a positive integer. Then $\mathbb{A}$ is primitive if and only if $\mathbb{A}^a$ is primitive.\end{proposition}

Let $\mathbb{A}$ be a nonnegative primitive tensor with order $m$ and dimension $n$, for any $j=j_1\in[n]$, by Proposition \ref{Pro25}, there exists a sequence $j_1, j_2, \cdots, j_{s+1}$ such that $j_k\in[n], j_k\ne j_{k+1}, 1\le k\le s$ and $(M(\mathbb{A}))_{j_{k+1}j_k}>0$.

\begin{definition}\label{defn28}
Let $\mathbb{A}$ be a nonnegative  tensor with order $m$ and dimension $n$,
if  $j_k\in[n]$ for $ 1\le k\le t$ and $(M(\mathbb{A}))_{j_{k+1}j_k}>0$ for $1\le k\le t-1$,
we say that $j_1\to j_2\to \cdots \to j_t$ is a $walk$ of $length$ $t-1$ of $M(\mathbb{A})$;
if $j_i\not= j_k$ for any $i,k\in [t]$ with $i\not=k$, then we say the walk  $j_1\to j_2\to \cdots \to j_t$
is a path  of $M(\mathbb{A})$;  if $j_i\not= j_k$ for any $i,k\in [t-1]$ with $i\not=k$ but $j_1=j_t$,
then we say the walk  $j_1\to j_2\to \cdots \to j_t$ is a cycle  of $M(\mathbb{A})$. \end{definition}

\begin{lemma} \label{lem29} Let $\mathbb{A}$ be a nonnegative  tensor with order $m$ and dimension $n$. Suppose that $t>1$ and  $j_1\to j_2\to \cdots \to j_t$ is a $walk$ of $M(\mathbb{A})$ , then  $(M(\mathbb{A}^{t-1}))_{j_tj_1}>0$.\end{lemma}

\begin{proof} We prove the assertion by induction on $t$.
Clearly, $t=2$ is obvious. Assume that the result holds for $t=k\ge2$, that is $(M(\mathbb{A}^{k-1}))_{j_kj_1}>0$. Since

$(M(\mathbb{A}^{k}))_{j_{k+1}j_1}=\sum\limits_{i_2,\cdots, i_m=1}^{n}a_{j_{k+1}i_2\cdots i_m}(M(\mathbb{A}^{k-1}))_{i_2j_1}\cdots(M(\mathbb{A}^{k-1}))_{i_mj_1}$

\hskip2.6cm $>a_{j_{k+1}j_k\cdots j_k}((M(\mathbb{A}^{k-1}))_{j_kj_1})^{m-1}$

\hskip2.6cm $=(M(A))_{j_{k+1}j_k}((M(\mathbb{A}^{k-1}))_{j_kj_1})^{m-1}>0,$

\noindent by  assumption, thus  the assertion holds for any integer $t>1$. We are done.\end{proof}

\begin{lemma}\label{le210} Let $\mathbb{A}$ be a nonnegative primitive tensor with order $m$ and dimension $n$. Then there exist an integer $j\in[n]$ and  an integer $l\in[n-1]$ such that  $(M(\mathbb{A}^l))_{jj}>0$.\end{lemma}

\begin{proof} By Proposition \ref{Pro25}, we may assume that $j_1\to j_2\to\cdots\to j_{n+1}$ is a walk of length $n$ of $M(\mathbb{A})$.
Since $j_k\in[n], 1\le k\le n+1$, there exist at least two integers $u$ and $v$ such that $1\le u<v\le n+1$ and $j_u=j_v$.
It follows from Lemma \ref{lem29} that $(M(\mathbb{A}^{v-u}))_{j_uj_u}>0$.

{\bf Case 1: } There exist $u, v$ such that $(u, v)\ne (1, n+1)$.

Then $v-u\le n-1$ and we are done by taking $j=j_u$ and $l=v-u$.

{\bf Case 2: } $(u, v)= (1, n+1)$.

Then $j_1, j_2, \ldots, j_n$ is a permutation of $1, 2, \ldots, n$ and $j_{n+1}=j_1$.
By (ii) of Proposition \ref{Pro25}, there exist an integer $i\in[n]$ and two integers $p\ne q$ such that $1\le p, q\le n$, $(M(\mathbb{A}))_{pi}>0$ and $(M(\mathbb{A}))_{qi}>0$. Take $i=j_t$ for some $t\in [n]$, and assume that $p=j_s\ne j_{t+1}$. Thus $s\in [n]$, and $t\leq n-1$ when $s=1$.

{\bf Subcase 2.1: } $p\in \{j_1, \cdots, j_{t-1}\}$.

Then $1\leq s\leq t-1$ and

$$p=j_s\to \cdots j_{t-1}\to j_t\to p=j_s$$
is a cycle of $M(\mathbb{A})$ with length $t+1-s\le n-1$.  We take $j=p$ and $l=t+1-s$.

{\bf Subcase 2.2: } $p= j_{t}$.

Then $s=t$ and $j_t\to p(=j_s)$ is a cycle of $M(\mathbb{A})$ with length $1$.  We take $j=j_t$ and $l=1$.

{\bf Subcase 2.3: } $p\in \{j_{t+2}, \cdots, j_{n}\}$.

Then $t+2\leq s\leq n$ and

$$j_1\to j_2\to\cdots\to j_t\to j_s(=p)\to\cdots\to j_n\to j_1$$
is a cycle of $M(\mathbb{A})$ and $t+n-s+1= n-(s-t-1)\le n-1$. In this case we take $j=j_1$ and $l=n+t+1-s$. This proves the lemma.\end{proof}

\begin{remark}\label{rem211}
By the proof of Lemma \ref{le210}, if  $\mathbb{A}$ is a nonnegative primitive tensor with order $m$ and dimension $n$,
then there exist an integer $t$ with $1\leq t\leq n-1$ and some integers
$j_1, j_2, \cdots, j_t\in [n]$ such that $j_1\to j_2\to \cdots\to j_t\to j_1$ is a cycle of length $t$ of $M(\mathbb{A})$,
 and for any $k\in [t]$, $(M(\mathbb{A}^t))_{j_kj_k}>0$.
\end{remark}

Note that by (\ref{eqm1}), (ii) of  Corollary \ref{Cor26} also holds when $\mathbb{A}$ is a nonnegative  tensor with order $m$ and dimension $n$. Therefore it makes sense to consider the primitive degree of some column of a tensor.

\begin{definition}\label{defn212}
 Let $\mathbb{A}$ be a nonnegative  tensor with order $m$ and dimension $n$. For a fixed integer $j\in[n]$, if there exists a positive integer $T$ such that
$$(M(\mathbb{A}^T))_{uj}>0, \hskip.2cm {\mbox for\,\, all } \, u, 1\le u\le n,$$
then $\mathbb{A}$ is called $j$-primitive and the smallest such integer $T$ is called the $j$-primitive degree of $\mathbb{A}$, denoted by $\gamma_j(\mathbb{A})$. \end{definition}

By Corollary \ref{Cor26} and the above definition, we have the following result.
\begin{proposition}\label{Pro213}
Let $\mathbb{A}$ be a nonnegative primitive  tensor with order $m$ and dimension $n$. Then
$$\gamma(\mathbb{A})=\max_{1\le j\le n}\{\gamma_j(\mathbb{A})\}.$$\end{proposition}

\section{Proof of the main results}

\hskip.6cm In this section, we will prove Theorem \ref{thmm}. We first prove the following special case of the theorem.

\begin{theorem}\label{thm31} Let $\mathbb{A}$ be a nonnegative primitive  tensor with order $m$ and dimension $n$. Suppose that  there is an integer $j\in[n]$ with
$(M(\mathbb{A}))_{jj}>0$, then $\gamma_j(\mathbb{A})\le n-1$. \end{theorem}

\begin{proof} Put
$$S_k=\{u\in[n], (M(\mathbb{A}^k))_{uj}\}>0, k=1, 2, \ldots.$$
Then $j\in S_1$ and there exists an integer $v\in[n]\setminus\{j\}$ such that $v\in S_1$ by (i) of Proposition \ref{Pro25}. Thus $|S_1|\geq 2$.

 Since $$(M(\mathbb{A}^{k+1}))_{uj}=\sum\limits_{i_2,\cdots, i_s=1}^{n}(\mathbb{A}^k)_{ui_2\cdots i_s}(M(\mathbb{A}))_{i_2j}\cdots(M(\mathbb{A}))_{i_sj}\ge (M(\mathbb{A}^k))_{uj}(M(\mathbb{A}))_{jj}^{s-1}, $$

\noindent we see that if $u\in S_k$, then $u\in S_{k+1}$. Therefore by induction on $k$, we can obtain that
$S_1\subseteq S_2 \subseteq \cdots \subseteq S_l \subseteq S_{l+1}\subseteq\cdots$.

Note that $S_k\subseteq[n]$ for any $k=1, 2, \cdots$, and $S_1\subseteq S_2 \subseteq \cdots \subseteq S_l \subseteq S_{l+1}\subseteq\cdots$,
 hence the sequence $S_1, S_2, \ldots, S_l, \ldots$  eventually  terminates.
 Let $l$ be the smallest positive integer such that $S_l=S_{l+1}$,  by (\ref{eqm1}) for $k=l, l+1$, we have
$$(M(\mathbb{A}^{l+1}))_{uj}=\sum\limits_{j_2,\cdots, j_m=1}^{n}a_{uj_2\cdots j_m}(M(\mathbb{A}^l))_{j_2j}\cdots(M(\mathbb{A}^l))_{j_mj}$$ and
$$(M(\mathbb{A}^{l+2}))_{uj}=\sum\limits_{j_2,\cdots, j_m=1}^{n}a_{uj_2\cdots j_m}(M(\mathbb{A}^{l+1}))_{j_2j}\cdots(M(\mathbb{A}^{l+1}))_{j_mj}.$$
It follows that $S_{l+1}=S_{l+2}=\cdots=S_k$ for all $k\ge l$.
Since $\mathbb{A}$ is a nonnegative primitive  tensor, hence $S_l=S_{\gamma(\mathbb{A})}=[n]$.

Now by the above argumnets and  the definition of $l$, we have
$$S_1\subsetneqq S_2\subsetneqq\cdots\subsetneqq S_{l-1}\subsetneqq S_l,$$ and thus $2\le |S_1|<|S_2|<\cdots<|S_l|=n$.
It follows that $l\le n-1$ and $\gamma_j(\mathbb{A})\le n-1$. 
\end{proof}

We also need the following lemmas.

\begin{lemma} \label{lem32} Let $\mathbb{A}$ be a nonnegative primitive  tensor with order $m$ and dimension $n$. Let $H=\{i_1, i_2, \ldots, i_s\}$ be the set of all elements $i\in[n]$ such that $i=k_1\to k_2\to\cdots\to k_t\to k_1=i$ is a cycle with length $t$ of $M(\mathbb{A})$ for some $t$ where $1\le t\le n-1$. Suppose there exists  a positive integer $j$  with $j\in [n]\setminus H$,  then  $1\le s\le n-1$ and there exist positive integers $i$ and $l$ such that $i\in H$, $1\le l\le n-s$ and $(M(\mathbb{A}^l))_{ij}>0$.\end{lemma}

\begin{proof}
It is obvious that $s\le n-1$,  and $s\geq 1$ by Lemma \ref{le210}.
By Proposition \ref{Pro25}, we have the following walk of length $n-s$ of $M(\mathbb{A})$ starting with  $j$:
$$j=j_1\to j_2\to\cdots\to j_{n-s+1}.$$
If there exists an integer $w,  1\le w\le n-s+1$ such that $j_w\in H$, then $(M(\mathbb{A}^{w-1}))_{j_wj}>0$ by Lemma \ref{lem29}, and we are done. Otherwise, we have
$$\{j_1, j_2, \cdots, j_{n-s+1}\}\cap H=\emptyset.$$
Since $n-s+1+|H|=n+1$, so there exist two positive integers $u$ and $v$ such that $1\le u<v\le n-s+1$, $j_u=j_v$,
and $j_u\to j_{u+1}\to \cdots\to j_{v-1}\to j_v=j_u$ is a cycle with length $v-u\leq n-s\le n-1$ of $M(\mathbb{A})$.
 It follows that $j_u\in H$, a contradiction. This proves the lemma.\end{proof}

Let $Z(\mathbb{A})$ be the tensor obtained by replacing all the nonzero entries of $\mathbb{A}$ by one. Then $Z(\mathbb{A})$ is called the zero-nonzero pattern (or simply the zero pattern) of $\mathbb{A}$. Let $a$ be a complex number, we define $Z(a)=1$ if $a\not=0$ and $Z(a)=0$ if $a=0$.

\begin{lemma}\label{lem33}
Let $\mathbb{A}$ be a nonnegative  tensor with order $m$ and dimension $n$ such that
  $a_{ii_2\cdots i_m}=0$ if $i_2\cdots i_m\not=i_2\cdots i_2$  for any $i\in [n]$.
  Then for any positive $r$, $Z(M(\mathbb{A}^r))=Z((M(\mathbb{A}))^r).$
\end{lemma}

\begin{proof}
We show the result by induction on $r$. Clearly, $r=1$ is obvious.

Assume that the assertion holds for $r=k\geq 1$, then for  any $i,j\in [n]$,

$Z[(M(\mathbb{A}^{k+1}))_{ij}]=Z[\sum\limits_{i_2,\cdots, i_m=1}^{n}a_{ii_2\cdots i_m}(M(\mathbb{A}^k))_{i_2j}\cdots(M(\mathbb{A}^k))_{i_mj}]$

\hskip2.9cm $=Z[\sum\limits_{i_2=1}^{n}(M(\mathbb{A}))_{ii_2}((M(\mathbb{A}^k))_{i_2j})^{m-1}]$

\hskip2.9cm $=Z[\sum\limits_{i_2=1}^{n}(M(\mathbb{A}))_{ii_2}(M(\mathbb{A}^k))_{i_2j}]$

\hskip2.9cm $=Z[\sum\limits_{i_2=1}^{n}Z[(M(\mathbb{A}))_{ii_2}]Z[(M(\mathbb{A}^k))_{i_2j}]]$

\hskip2.9cm $=Z[\sum\limits_{i_2=1}^{n}Z[(M(\mathbb{A}))_{ii_2}]Z[((M(\mathbb{A}))^k)_{i_2j}]]$

\hskip2.9cm $=Z[\sum\limits_{i_2=1}^{n}Z[(M(\mathbb{A}))_{ii_2}((M(\mathbb{A}))^k)_{i_2j}]]$

\hskip2.9cm $=Z[((M(\mathbb{A}))^{k+1})_{ij}]$

Thus $Z(M(\mathbb{A}^{k+1}))=Z((M(\mathbb{A}))^{k+1}).$
\end{proof}

\begin{corollary}\label{cor34}
Let $\mathbb{A}$ be a nonnegative primitive  tensor with order $m$ and dimension $n$ such that
  $a_{ii_2\cdots i_m}=0$ if $i_2\cdots i_m\not=i_2\cdots i_2$  for any $i\in [n]$.
  If $M(\mathbb{A})$ is primitive, then
   $\gamma(\mathbb{A})=\gamma(M(\mathbb{A})).$
\end{corollary}

\noindent {\bf Proof of Theorem \ref{thmm}:} Let $H=\{i_1, i_2, \ldots, i_s\}$ be the set of all elements $i\in[n]$ such that $i=k_1\to k_2\to\cdots\to k_t\to k_1=i$ is a cycle with length $t$ of $M(\mathbb{A})$ and $1\le t\le n-1$.

{\bf Case 1: } $s=n$.

Then for any $j\in [n]$, there exists an integer, say,  $t_j$, such that there exists a cycle $j=k_1\to k_2\to \cdots k_{t_j}\to k_1=j$ with length $t_j$ where  $t_j\in [n-1]$,
so $(M(\mathbb{A}^{t_j}))_{jj}>0$. Note that $\mathbb{A}^{t_j}$ is primitive by Proposition \ref{Pro27} and $\mathbb{A}$ is primitive,
we have $\gamma_j(\mathbb{A}^{t_j})\leq n-1$ by Theorem \ref{thm31}. Hence
$$\gamma_j(\mathbb{A})\le t_j\gamma_j(\mathbb{A}^{t_j})\le (n-1)^2.$$
Thus $\gamma(\mathbb{A})\le \max_{j\in [n]}\{\gamma_j(\mathbb{A})\}\le (n-1)^2$ by Proposition \ref{Pro213}.

{\bf Case 2: } $1\leq s \leq n-1$.

Then there exists at least an integer $j\in [n]\backslash H$.

{\bf Subcase 2.1: } $j\in H$.

Similar to  Case 1, for any $j\in H$, there exists an integer, say,  $t_j$, such that there exists a cycle $j=k_1\to k_2\to \cdots k_{t_j}\to k_1=j$ with length $ t_j$ where $t_j\leq s$, so $(M(\mathbb{A}^{t_j}))_{jj}>0$.
Note that $\mathbb{A}^{t_j}$ is primitive by Proposition \ref{Pro27} and $\mathbb{A}$ is primitive,
we have $\gamma_j(\mathbb{A}^{t_j})\leq n-1$ by Theorem \ref{thm31}. Hence
$$\gamma_j(\mathbb{A})\le t_j\gamma_j(\mathbb{A}^{t_j})\le t_j(n-1)\le s(n-1).$$
Let $T=\max_{j\in H}\{\gamma_j(\mathbb{A})\}$, then $T\le s(n-1)$ and $(M(\mathbb{A}^T))_{uj}>0$ for any $j\in H$ and all $u\in[n]$.

{\bf Subcase 2.2: } $j\in [n]\backslash H$.

By Lemma \ref{lem32}, there exist positive integers $i$ and $l$ such that $i\in H$, $1\le l\le n-s$ and $(M(\mathbb{A}^l))_{ij}>0$.  For any $u\in[n]$, since
$$(M(\mathbb{A}^{T+l}))_{uj}=\sum\limits_{j_2,\cdots, j_s=1}^{n}(\mathbb{A}^T)_{uj_2\cdots j_s}(M(\mathbb{A}^l))_{j_2j}\cdots(M(\mathbb{A}^l))_{j_sj}
\ge (M(\mathbb{A}^T))_{ui}((M(\mathbb{A}^l))_{ij})^{s-1}>0, $$
 then $\gamma_j(\mathbb{A})\le T+l\le s(n-1)+n-s\le (n-1)^2+1$.

 Thus combing Subcase 2.1 and Subcase 2.2, we have $\gamma(\mathbb{A})\le \max_{j\in [n]}\{\gamma_j(\mathbb{A})\}\le (n-1)^2+1$ by Proposition \ref{Pro213}.

Combing the above arguments, $\gamma(\mathbb{A})\le \max\{(n-1)^2, (n-1)^2+1\}= (n-1)^2+1$.

Let $M_1=\left(
                 \begin{array}{ccccc}
                   0 & 0  & \cdots & 1 & 1 \\
                   1 & 0   & \cdots &  0&  0\\
                    0 & 1  & \cdots &  0& 0  \\
                    0 & 0 & \ddots &  0&0  \\
                   0 &0 & \cdots & 1 & 0 \\
                 \end{array}
               \right)
$. It is well known that $M_1$ is primitive, and $\gamma(M_1)=(n-1)^2+1$.

Let  $\mathbb{A}$ be a nonnegative primitive  tensor with order $m$ and dimension $n$ such that
  $a_{ii_2\cdots i_m}=0$ if $i_2\cdots i_m\not=i_2\cdots i_2$  for any $i\in [n]$, and $M(\mathbb{A})=M_1$.
  Then by Corollary \ref{cor34}, we have $\gamma(\mathbb{A})=(n-1)^2+1.$
\qed

\begin{remark}\label{rem35}
It is well known $\gamma(\mathbb{A})\leq (n-1)^2+1$ for $m=2$ ($\mathbb{A}$ is a matrix). It implies that the upper bound on the primitive degree of primitive tensors
and the upper bound on the primitive index of primitive matrices are coincident.
\end{remark}


\end{document}